\documentclass[reqno]{amsart}

\usepackage{amsmath,amssymb,amsthm,amsfonts}
\usepackage{mathrsfs}
\usepackage[colorlinks=true,linkcolor=blue,citecolor=blue]{hyperref}
\usepackage[all,cmtip]{xy} % 交换图表

\numberwithin{equation}{section}
\allowdisplaybreaks

\newcommand{\cF}{\mathcal{F}}

\newcommand{\cO}{\mathcal{O}}

\newcommand{\bC}{\mathbb{C}}
\newcommand{\bF}{\mathbb{F}}
\newcommand{\bP}{\mathbb{P}}
\newcommand{\bQ}{\mathbb{Q}}
\newcommand{\bZ}{\mathbb{Z}}

\DeclareMathOperator{\Supp}{Supp}

\theoremstyle{plain}
\newtheorem{theorem}{Theorem}[section]
\newtheorem{lemma}[theorem]{Lemma}
\newtheorem{proposition}[theorem]{Proposition}
\newtheorem{corollary}[theorem]{Corollary}

\theoremstyle{definition}
\newtheorem{definition}[theorem]{Definition}

\theoremstyle{remark}
\newtheorem{remark}[theorem]{Remark}

\title{Noether-type inequalities for big divisors via control of the negative part}

\author{Shi Xu}

\address[Shi Xu]{Yau Mathematical Sciences Center, Tsinghua University, Beijing, 100084, China}

\email{shixumath@163.com}

\date{}



\begin{document}

\begin{abstract}
Let $X$ be a smooth projective surface over $\bC$ and $D$ a big divisor with Zariski decomposition $D=P+N$.
We study the relationship between the volume $\mathrm{vol}(D)=P^2$ and the dimension $h^0(D)$.

We introduce a numerical invariant $\mathfrak{C}(N)$ depending only on the negative part $N$,
which provides a universal baseline control for $\mathrm{vol}(D)$.
This allows us to establish Noether-type inequalities relating $\mathrm{vol}(D)$ and $h^0(D)$,
where all correction terms are explicitly governed by $\mathfrak{C}(N)$.

Our results recover and unify several classical inequalities on surfaces,
and apply in particular to adjoint divisors and foliations.

We further obtain lower bounds for $\mathrm{vol}(D)$ in terms of the ps-index $\iota(D)$,
with applications to foliated surfaces.
\end{abstract}

\maketitle

\setcounter{tocdepth}{1}  % 只显示到 section
\tableofcontents

\section{Introduction}

Let $X$ be a smooth projective surface over $\bC$ and $D$ a big divisor.
Understanding the relationship between the volume $\mathrm{vol}(D)=P^2$
and the dimension $h^0(D)$ of the space of global sections of $\mathcal{O}_X(D)$
is a classical problem in algebraic geometry, dating back to
Noether inequalities for surfaces of general type (cf.~\cite{Noe70}).

\medskip
We are mainly interested in the case $h^0(D)\ge 2$, which ensures that the complete linear system $|D|$ defines a non-trivial rational map.

A fundamental difficulty arises from the discrepancy between the 
Zariski decomposition
\[
D=P+N
\]
and the linear system decomposition
\[
|D|=|M|+Z.
\]
The first encodes a numerical decomposition of $D$, 
while the second reflects the effective geometry of its linear system. 
Relating these two structures is therefore a central issue.
\medskip

In this paper, we introduce a numerical invariant $\mathfrak{C}(N)$, depending only on the negative part $N$, 
which provides a uniform control on this discrepancy. 
This allows us to establish Noether-type inequalities for arbitrary big divisors on smooth surfaces.

Now we state the main results.

\subsection{Main results}
Our main results are divided according to the dimension
of the image of the rational map $\phi_{|D|}$.
To simplify the statements, we fix the following notation and assumptions.

\medskip

\noindent\textbf{(A).}
Let $X$ be a smooth projective surface over $\mathbb{C}$, and let $D$ be a big divisor on $X$ with $h^0(D)\ge 2$.
We consider the following associated objects:

\begin{itemize}
\item The Zariski decomposition $D = P + N$, where $P$ is nef and $N$ is the negative part;

\item The decomposition of the linear system $|D| = |M| + Z$, where $M$ is the movable part and $Z$ is the fixed part;

\item The linear system $|D|$ induces the rational map $\phi_{|D|}=\phi_{|M|}: X \dashrightarrow \mathbb{P}^d$, where $d=h^0(D)-1$;

\item $\mathfrak{C}(N)$ denotes the invariant depending only on $N$ defined in Section~\ref{sec:N-nef}.
\end{itemize}
\medskip

\begin{theorem}\label{thm:main}
Throughout this theorem, we work under the assumptions in {\rm (A)}.

\begin{enumerate}
\item[(1)] If $\dim{\rm Im}\phi_{|D|}=1$, then  
\begin{equation}\label{ineq:P^2-intro-dim1}
\mathrm{vol}(D)
\ge \frac{(h^0(D)-1)^2}{h^0(D)-1+\mathfrak{C}(N)},
\end{equation}
with equality only if ${\rm Supp}(Z)={\rm Supp}(N)$ 
and $\phi_{|D|}$ induces a fibration $f:X\to \bP^1$ with connected fibres such that $D\cdot F=1$, where $F$ is a general fibre of $f$.

\item[(2)] If $\dim{\rm Im}\phi_{|D|}=2$, then 
\begin{equation}\label{ineq:P^2-intro-dim2-general-1st}
\mathrm{vol}(D) \ge h^0(D)-2,
\end{equation}
with equality if and only if $Z=N$ and $\phi_{|D|}$ is a birational morphism
onto  a normal rational surface $\Sigma\subset\bP^d$ of degree $d-1$ (see Proposition~\ref{prop:degS=n-1} for precise description).

Moreover, if the equality in \eqref{ineq:P^2-intro-dim2-general-1st} does not hold, then
\begin{equation}\label{ineq:P^2-intro-dim2-general-2nd}
\mathrm{vol}(D) \ge h^0(D)-2+\frac{1}{1+\mathfrak{C}(N)}.
\end{equation}

\item[(3)] If $\dim{\rm Im}\phi_{|D|}=2$ and $\kappa(X)\geq0$, then 
\begin{equation}\label{ineq:P^2-intro-dim2-not-rule-1st}
\mathrm{vol}(D) \ge 2h^0(D)-4,
\end{equation}
with equality if and only if $Z=N$ and either $\phi_{|D|}$ is a birational morphism onto a surface (birational to a K3 surface) of degree $2d-2$ of $\bP^d$,  
or $\phi_{|D|}$ is a finite morphism of degree $2$ onto a normal rational surface $\Sigma \subset \bP^d$ of degree $d-1$ (see Proposition~\ref{prop:degS=n-1} for a precise description of $\Sigma$).

Moreover, if the equality in \eqref{ineq:P^2-intro-dim2-not-rule-1st} does not hold, then

\begin{equation}\label{ineq:P^2-intro-dim2-not-rule-2nd}
\mathrm{vol}(D) \ge 2h^0(D)-4+\frac{1}{1+\mathfrak{C}(N)}.
\end{equation}
\end{enumerate}
\end{theorem}

\medskip

A conceptual feature of our approach is that the discrepancy between the
Zariski decomposition and the linear system decomposition can be separated
from the geometry of the movable linear system.

More precisely, in Section~\ref{sec:P+NvsM+Z} we establish comparison inequalities between the
positive part $P$ and the movable part $M$, where the discrepancy is uniformly
controlled by the invariant $\mathfrak{C}(N)$ depending only on the negative part $N$.
These comparison results are then combined with classical geometric properties
of the movable linear system $|M|$ and the induced rational map
$\phi_{|M|}$,
leading to the Noether-type inequalities proved in Section~\ref{sec:proof}.

\subsection{Applications: divisors with controlled negative part}
The effectiveness of the above inequalities depends on the invariant $\mathfrak{C}(N)$.
A key point is that in many geometric situations,
the negative part $N$ belongs to a restricted class,
so that $\mathfrak{C}(N)$ admits a uniform bound independent of the individual divisor $D$.
Consequently, the general inequalities obtained in the previous subsection
become effective uniform lower bounds for $\mathrm{vol}(D)$.

We now discuss several natural classes arising in surface theory
where such uniform bounds can be established.
\medskip

\paragraph{\bf Big and nef divisors.}
If $D$ is nef and big, then $\mathfrak{C}(N)=0$.
In this case, the baseline estimate becomes classical,
recovering the Noether inequality for minimal surfaces of general type
(see also \cite{Shin08}).

\medskip

\paragraph{\bf Adjoint divisors.}
Let $D=K_X+L$ with $L\equiv A+B$, where $A$ and $K_X+A$ are nef and
$B=\sum b_i C_i$ with $b_i\in[0,1]$.
Then $\mathfrak{C}(N)\le2$ (cf.~Proposition~\ref{prop:Zari-decom-of-K_X+L}).
Hence all divisors in this class admit a uniform baseline lower bound for $\mathrm{vol}(D)$.

This framework includes canonical divisor of log surfaces \cite{TZ92}
and of normal KSBA stable surfaces via resolution \cite{Che23},
recovering and unifying several known inequalities.

In \cite{Che23}, sharp examples are constructed showing that the bound $\mathfrak{C}(N)\le 2$ is optimal.
Our approach provides a conceptual explanation of these extremal examples via the invariant $\mathfrak{C}(N)$,
and yields an alternative proof of most inequalities in \cite[Theorem~1.1]{Che23}.

\medskip

\paragraph{\bf Hirzebruch--Jung chains.} 
If $\Supp(N)=\cup_{i=1}^r\Gamma_i$ is a Hirzebruch--Jung chain with $N\cdot\Gamma_1=-1$ and $N\cdot\Gamma_i=0$ for $i\ge2$, 
then $\mathfrak{C}(N)\le1$ (cf.~Proposition~\ref{prop:e_N(A)<=1}). 
This situation arises, for instance, in canonical divisors of relatively minimal foliations (cf.~\cite{Bru15,McQ08}).

In the general  foliation case,
the examples in \cite{LT24} show that the bound $\mathfrak{C}(N)\le 1$ is sharp,
and our framework provides an alternative proof of the corresponding inequalities
(cf.~\cite[Proposition~8.1(i)(ii) and Proposition~8.2(i)]{LT24}).

On the other hand, in the algebraically integrable case treated in \cite{LT24},
additional geometric structures impose further constraints on the linear system,
leading to sharper bounds that are not captured solely by $\mathfrak{C}(N)$.

\medskip

This shows that $\mathfrak{C}(N)$ controls the universal part of the lower bound,
while finer geometric properties of the linear system $|D|$
lead to sharper inequalities in special situations.

\subsection{Volume and ps-index of a big divisor}
For a big divisor $D$ on a smooth surface $X$, one does not necessarily have $h^0(D)\geq2$.
We therefore introduce the following invariant:
\[
\iota(D):=\min\{\,m\in\mathbb{Z}_{>0}\mid h^0(mD)\ge 2\,\}.
\]
We refer to $\iota(D)$ as the \emph{pluricanonical section index} (ps-index) of $D$.
The terminology is borrowed from \cite{CC14}, where the authors study the ps-index of canonical divisors on higher-dimensional algebraic varieties.

\begin{corollary}\label{coro:vol-ps-index}
Let $D$ be a big divisor on a smooth surface $X$, and let $\iota(D)$ be its ps-index.
Then
\[
\mathrm{vol}(D)
\;\ge\;
\frac{1}{\iota(D)^2}
\cdot
\frac{1}{1+\mathfrak{C}(N)\,\iota(D)}.
\]
\end{corollary}

\begin{remark}
This shows that the volume of $D$ is bounded from below in terms of its ps-index,
with an explicit dependence on the negative part $N$.
\end{remark}
\medskip

Finally, we consider the case where $D=K_{\cF}$, with $\cF$ a relatively minimal foliation of general type on a smooth surface $X$.
In this situation, one has $\mathfrak{C}(N)\leq 1$. 
Denote ${\rm vol}(\cF):={\rm vol}(K_{\cF})$ and $\iota(\cF):=\iota(K_{\cF})$.
Then the above corollary yields
\begin{equation}
{\rm vol}(\cF)\geq \frac{1}{\iota(\cF)^2}\cdot\frac{1}{1+\iota(\cF)}.
\end{equation}

This is closely related to a question of Cascini \cite{Cas21} on the boundedness of foliated varieties,
concerning the existence of uniform constants controlling both the birationality of pluricanonical systems
and the volume of $K_{\cF}$.

It was shown in \cite{Lu25} that such a uniform bound for birationality does not exist in general when $\mathrm{rank}\,\cF<\dim X$.
On the other hand, the existence of a uniform bound for the ps-index $\iota(\cF)$ remains open.

\medskip

\noindent\textbf{Question.}
Does there exist a constant $\iota_{n,r}>0$, depending only on $n=\dim X$ and $r={\rm rank}\,\cF$, where $r<n$, such that
\[
h^0(mK_{\cF})\geq 2 \quad \text{for some } m\le \iota_{n,r} \ ?
\]

\medskip

The paper is organized as follows.
In Section~\ref{sec:pre}, we collect some standard facts and introduce the basic notions
used throughout the paper.
In Section~\ref{sec:N-nef}, we introduce the invariant $\mathfrak{C}(N)$
and establish its basic properties.
In Section~\ref{sec:P+NvsM+Z}, we derive a fundamental identity relating
$P$ and $M$, together with uniform lower bounds.
Finally, in Section~\ref{sec:proof}, we combine these results with
the geometry of linear systems to prove the main theorems.

\bigskip

\noindent{\bf Notation.} Unless otherwise stated, all notations and conventions are standard in algebraic geometry. 
Linear equivalence is denoted by $\sim$, and numerical equivalence is denoted by $\equiv$.

\section{Preliminaries}\label{sec:pre}

Throughout the paper, let $X$ be a smooth projective surface over $\bC$,
and let $D$ be a divisor on $X$.

\begin{definition}
A divisor $D$ is called \emph{numerically effective} (or simply \emph{nef}) if $D\cdot C \ge 0$ for every curve $C$ on $X$.  
It is called \emph{pseudo-effective} if $D\cdot H \ge 0$ for every ample divisor $H$ on $X$.  
Finally, $D$ is called \emph{big} if there exists a constant $\alpha \in \bQ_{>0}$ such that 
$
h^0(X, nD) \ge \alpha n^2
$
for all sufficiently large integers $n$.
\end{definition}

Let $N$ be an effective $\bQ$-divisor on $X$.
We introduce the following terminology.

\begin{definition}
 A divisor $D$ is called \emph{$N$-nef} if $D\cdot \Gamma \ge 0$ for every irreducible component $\Gamma$ of $N$.  

 Let $D_1$ and $D_2$ be two $\bQ$-divisors on $X$.  
We say that $D_1$ and $D_2$ are \emph{numerically $N$-equivalent}, denoted by $D_1 \equiv_N D_2$,  
if $D_1\cdot \Gamma = D_2\cdot \Gamma$ for every irreducible component $\Gamma$ of $N$.
\end{definition}

\begin{theorem}[Zariski decomposition~{\cite{Zar62}}]
Let $D$ be a pseudo-effective divisor on $X$.
Then there exist unique $\bQ$-divisors $P$ and $N$ on $X$ such that 
\[ D = P + N, \]
satisfying the following conditions:
\begin{enumerate}
\item $N=0$ or the intersection matrix of the irreducible components of $N$ is negative definite;
\item $P$ is nef and $N$ is effective;  
\item $P\cdot \Gamma=0$ for each irreducible component $\Gamma$ of $N$. 
\end{enumerate}
We call $P$ (resp. $N$) the \emph{positive} (resp. \emph{negative}) part of $D$. 
\end{theorem}

For a pseudo-effective divisor $D$, we define its \emph{volume} by 
\[
{\rm vol}(D):=\limsup_{n\to+\infty}\frac{h^0(X,nD)}{n^2/2}.
\]
It is well known that
\[
{\rm vol}(D)=P^2.
\]
In particular, if $D$ is big, then ${\rm vol}(D)>0$.

\begin{lemma}\label{lem:neg-def}
Suppose $D=\sum_{i=1}^n a_i C_i$ is a $\bQ$-divisor such that the intersection matrix 
$(C_i\cdot C_j)_{1\leq i,j\leq n}$ is negative definite.
\begin{enumerate}
\item If $D\cdot C_i \leq 0$ for all $i=1,\dots,n$, then $D \geq 0$.
\item If $E$ is an effective $\bQ$-divisor and $(E-D)\cdot C_j \leq 0$ for all $j=1,\dots,n$, then $E-D \geq 0$.
\end{enumerate}
\end{lemma}

\begin{proof}
The statements follow from {\cite[Lemma~14.9, Lemma~14.15]{Bad01}}.
\end{proof}

\begin{lemma}\label{lem:degD>=h0(D)-1-curve}
Let $Y$ be a smooth projective curve of genus $g(Y)$, and let $D$ be an effective divisor on $Y$.
\begin{enumerate}
\item If $\deg D > 2g(Y) - 2$, then 
\[
h^0(D) = \deg D - g(Y) + 1.
\]
\item \emph{(Clifford)} If $0 < \deg D \leq 2g(Y) - 2$, then 
\[
h^0(D) \leq \tfrac{1}{2}\deg D + 1.
\]
\end{enumerate}
Consequently, one always has
\[
\deg D \geq h^0(D) - 1.
\]
If $\deg D \geq 1$, then equality $\deg D = h^0(D) - 1$ holds if and only if $Y \cong \bP^1$.
\end{lemma}

\begin{proof}
See \cite[Ch.~IV, Thm.~1.3 and Thm.~5.4]{Har77}.
\end{proof}

\begin{proposition}[{\cite[Lem.~1.4]{Bea79}}]\label{Prop:Bea79Lem1.4}
Let $\Sigma$ be an irreducible non-degenerate surface in $\bP^d$. 
\begin{enumerate}
\item $\deg \Sigma \geq d - 1$.
\item If moreover $\kappa(\Sigma)\geq0$, then $\deg \Sigma \geq 2d - 2$.
\end{enumerate}
\end{proposition}

The following proposition describes all non-degenerate surfaces
of minimal degree.

\begin{proposition}\label{prop:degS=n-1}
Suppose $\Sigma \subset \bP^d$ is a non-degenerate surface of degree $d-1$.
Let $\rho:Y\to \Sigma$ be the minimal resolution of singularities of $\Sigma$, 
and let $M_0=\rho^*(H)$ be the pullback of a hyperplane section $H$ of $\Sigma$.

Then the triple $(Y,\Sigma,M_0)$ belongs to one of the following cases:
\smallskip

\begin{enumerate}
\item $d=2$, $Y=\Sigma=\bP^2$, and $M_0=L$, where $L$ is a general line in $\bP^2$.

\item $d=5$, $Y=\Sigma=\bP^2$, and $M_0=2L$, where $L$ is a general line in $\bP^2$.

\item $d\geq3$, $Y=\Sigma=\bF_e$ with $0\leq d-e-3\equiv0\pmod{2}$, and 
\[
M_0=C_e+\tfrac{1}{2}(d+e-1)F.
\]

\item[\rm(4)] $d\geq3$, $\Sigma$ is a cone over a rational curve of degree $d-1$ in $\bP^{d-1}$, 
$Y=\bF_e$, and 
\[
M_0=C_e+(d-1)F.
\]
\end{enumerate}

Here $\bF_e:=\bP_{\bP^1}\!\left(\cO_{\bP^1}\oplus\cO_{\bP^1}(e)\right)$ is the Hirzebruch surface of degree $e$, 
$F$ is a fibre of the $\bP^1$-fibration $\pi:\bF_e\to\bP^1$, and $C_e$ is a minimal section with $C_e^2=-e$.
\smallskip

In particular, $\Sigma$ is a normal rational surface.
\end{proposition}
\begin{proof}
See \cite[Theorem~7]{Nag60} or \cite[Lemma~2.2 and~2.3]{TZ92}.
\end{proof}

\begin{proposition}[{\cite[Rem.~1.5]{Bea79}}]\label{prop:degS=2d-2}
Suppose $\Sigma \subset \bP^d$ is a non-degenerate surface of degree $2d-2$.
Then $\Sigma$ is birational to a K3 surface.
\end{proposition}

\section{Divisors nef along a negative definite cycle}\label{sec:N-nef}
The purpose of this section is to introduce a numerical invariant 
associated to a negative-definite effective $\bQ$-divisor $N$, which provides a 
uniform control on the discrepancy between the Zariski decomposition 
and the linear system decomposition (see Section~\ref{sec:P+NvsM+Z}).
\medskip

Let $N=\sum_{i=1}^r \gamma_i \Gamma_i$ be an effective $\bQ$-divisor such that the intersection matrix $(\Gamma_i \cdot \Gamma_j)$ is negative definite.

Let $A$ be an $N$-nef divisor, i.e., $A \cdot \Gamma_i \ge 0$ for all $i$.

Associated to $A$, we define a divisor supported on $\Supp(N)$ as follow:
\begin{equation}\label{equ:def-E_N(A)}
E_N(A)=\sum_i b_i \Gamma_i
\end{equation}
 is the unique $\bQ$-divisor supported on $\Supp(N)$ such that
\[
E_N(A)\cdot \Gamma_i = - A \cdot \Gamma_i \quad \text{for all } i.
\]
By Lemma~\ref{lem:neg-def}, $E_N(A)$ is effective.

Intuitively, $E_N(A)$ is the correction divisor supported on $N$
which makes
\[
A^*:=A+E_N(A)
\]
numerically orthogonal to $N$.

\medskip

We now introduce the following numerical invariant.

\begin{definition}\label{def:lambda_N(A)}
We define
\[
\lambda_N(A) := \min \{ x \ge 0 \mid (xE_N(A) - N)\cdot A = 0 \}.
\]
\end{definition}

If $A\cdot N=0$, then $\lambda_N(A)=0$. 
If $A\cdot N>0$, then necessarily $A\cdot E_N(A)>0$, and
\[
\lambda_N(A)=\frac{A\cdot N}{A\cdot E_N(A)}.
\]

\medskip

The following lemma shows that $\lambda_N(A)$ is uniformly bounded 
in terms of the numerical data of $N$.

\begin{lemma}\label{lem:e_N_bound}
We have
\[
\lambda_N(A) \le \max_i \{ \gamma_i (-\Gamma_i^2) \}.
\]
\end{lemma}

\begin{proof}
Set 
\[
\lambda_0 := \max_i \{ \gamma_i (-\Gamma_i^2) \}.
\]
By definition of $\lambda_N(A)$, it suffices to show
\[
(\lambda_0 E_N(A) - N)\cdot A \ge 0.
\]

If $A \cdot N = 0$, then $E_N(A)=0$ and there is nothing to prove.  
Thus we may assume $A \cdot N > 0$.

For each $i$, we have
\[
\left( E_N(A) - \sum_j \frac{A \cdot \Gamma_j}{-\Gamma_j^2}\Gamma_j \right)\cdot \Gamma_i
=-A\cdot \Gamma_i - \sum_j \frac{A \cdot \Gamma_j}{-\Gamma_j^2}\Gamma_j\cdot \Gamma_i
=- \sum_{j\neq i} \frac{A \cdot \Gamma_j}{-\Gamma_j^2}\Gamma_j\cdot \Gamma_i
\le 0.
\]
By Lemma~\ref{lem:neg-def}, it follows that
\begin{equation}\label{equ:E_N(A)>=}
E_N(A) \ge \sum_i \frac{A \cdot \Gamma_i}{-\Gamma_i^2}\Gamma_i.
\end{equation}
Multiplying by $\lambda_0$, we obtain
\[
\lambda_0 E_N(A) 
\ge \sum_i \lambda_0 \frac{A \cdot \Gamma_i}{-\Gamma_i^2}\Gamma_i
\ge \sum_i (A \cdot \Gamma_i)\gamma_i \Gamma_i,
\]
where the last inequality follows from the definition of $\lambda_0$.

Note that $A$ is an $N$-nef divisor. Intersecting with $A$, we get
\[
\lambda_0 E_N(A)\cdot A\geq\sum_i\gamma_i(A\cdot\Gamma_i)^2\geq \sum_i\gamma_i A\cdot\Gamma_i=N\cdot A,
\]
which completes the proof.
\end{proof}
\medskip

The above lemma allows us to introduce a global invariant associated to $N$.

\begin{definition}\label{def:c(N)}
We define
\[
\mathfrak{C}(N):=\sup\{\lambda_N(A)\mid A \text{ is $N$-nef}\}.
\]
\end{definition}

By Lemma~\ref{lem:e_N_bound}, the invariant $\mathfrak{C}(N)$ is bounded 
and depends only on the numerical data of $N$.
This invariant provides a uniform upper bound for  
$\lambda_N(A)$ as $A$ varies among $N$-nef divisors.
\medskip

The following consequence is immediate from the definition.
\begin{lemma}\label{lem:lambda-scaling}
Let $N$ be a negative-definite effective $\bQ$-divisor and $A$ is an $N$-nef divisor.
\begin{enumerate}
\item If $A\equiv_N nF$ for a divisor $F$ and some $n\in\bZ_{>0}$, then 
\[
\lambda_N(A)\leq \frac{\mathfrak{C}(N)}{n}.
\]
\item If $x\in\bQ_{>0}$, then $\lambda_{xN}(A)=x\,\lambda_N(A)$. Moreover, $\mathfrak{C}(xN)=x\,\mathfrak{C}(N)$.
\end{enumerate}

\end{lemma}

\bigskip

We now estimate $\mathfrak{C}(N)$ in two classes of examples
arising naturally in surface theory.

\subsection{The case of adjoint divisors}

Let $D=K_X+L$ be an adjoint divisor, where 
\[
L \equiv A+B,
\]
with $A$ a nef $\bQ$-divisor such that $K_X+A$ is nef, and 
$B=\sum_i b_i C_i$ an effective $\bQ$-divisor with $b_i\in[0,1]$.

Since $K_X+L \equiv (K_X+A)+B$
with $K_X+A$ nef and $B\ge 0$, the divisor $K_X+L$ is pseudo-effective, hence admits a Zariski decomposition.

\begin{proposition}\label{prop:Zari-decom-of-K_X+L}
Under the above assumptions, let
\[
K_X+L = P+N
\]
be the Zariski decomposition. Then the negative part $N$ satisfies
\[
N=\sum_i \beta_i C_i,\qquad \beta_i\in[0,b_i]\cap\bQ.
\]
Moreover, if $\beta_j>0$, then
$p_a(C_j)=0$ and $\beta_j(-C_j^2)\leq 2b_j\leq 2$.
In particular,
\[
\mathfrak{C}(N)\le \max_j\{2b_j\}\le 2.
\]
\end{proposition}

\begin{proof}
For any irreducible component $C$ of ${\rm Supp}(N)$, we have
\[
(B-N)\cdot C
= B\cdot C-(K_X+L)\cdot C
= -(K_X+A)\cdot C \le 0,
\]
where the last inequality follows from the nefness of $K_X+A$.
Therefore, Lemma~\ref{lem:neg-def} implies that $N\le B$.

Now assume that $\beta_j>0$. Then
\begin{align*}
0 
&= P\cdot C_j \\
&= (K_X+A)\cdot C_j + \sum_i (b_i-\beta_i) C_i\cdot C_j \\
&\geq b_j K_X\cdot C_j + (b_j-\beta_j) C_j^2 \\
&= b_j(2p_a(C_j)-2) + \beta_j(-C_j^2),
\end{align*}
where we used that $A$ and $K_X+A$ are nef and $C_i\cdot C_j\geq 0$ for $i\neq j$.
Thus,
\[
0<\beta_j(-C_j^2)\leq b_j(2-2p_a(C_j))\leq 2b_j,
\]
which implies $p_a(C_j)=0$ and $\beta_j(-C_j^2)\leq 2b_j\leq 2$.

The last statement follows from Lemma~\ref{lem:e_N_bound}.
\end{proof}

\subsection{The case of Hirzebruch--Jung chains}

Assume that $\Gamma_1+\cdots+\Gamma_r$ is a Hirzebruch--Jung chain 
and $N=\sum_{i=1}^r\gamma_i\Gamma_i$ is the effective $\bQ$-divisor
such that 
\[
N\cdot \Gamma_1=-1,\qquad N\cdot \Gamma_i=0 \quad \text{for } i\ge2.
\]

\begin{proposition}\label{prop:e_N(A)<=1}
Under the above assumptions, we have 
\[
\mathfrak{C}(N)\le1.
\]
\end{proposition}

\begin{proof}
Let $A$ be an $N$-nef divisor.

Choose $t \in \{0,\dots,r\}$ such that
\[
A\cdot \Gamma_i = 0 \text{ for } i \le t, \qquad A\cdot \Gamma_{t+1} > 0,
\]
with the convention that $t=0$ if $A\cdot \Gamma_1>0$ and $t=r$ if $A\cdot \Gamma_i=0$ for all $i$.

Let $N'=\sum_{i=1}^t \gamma'_i \Gamma_i$ be the effective divisor satisfying
\[
N'\cdot \Gamma_1=-1,\qquad N'\cdot \Gamma_i=0 \quad \text{for } i\ge2,
\]
and set $N'=0$ if $t=0$.

\medskip

\noindent\textbf{Claim.} $Z:=E_N(A)+N'-N\ge0$.

\medskip

By Lemma~\ref{lem:neg-def}, it suffices to check that $Z\cdot \Gamma_i\le0$ for all $i$.

If $t=r$, then $Z=0$.

If $t=0$, then $Z=E_N(A)-N$, and since $A\cdot \Gamma_1>0$, we have
\[
Z\cdot \Gamma_1 = -A\cdot \Gamma_1 +1 \le 0,
\qquad
Z\cdot \Gamma_i = -A\cdot \Gamma_i \le 0 \ (i\ge2).
\]

If $1\le t\le r-1$, then
\[
Z\cdot \Gamma_i=
\begin{cases}
0,& i\le t,\\
-A\cdot \Gamma_{t+1} + N'\cdot \Gamma_{t+1} \le -1 + \gamma'_t < 0,& i=t+1,\\
-A\cdot \Gamma_i \le 0,& i\ge t+2.
\end{cases}
\]

This proves the claim.

\medskip

Therefore,
\[
(E_N(A)-N)\cdot A
= (E_N(A)+N'-N)\cdot A \ge 0,
\]
and hence $\lambda_N(A)\le1$. Therefore, $\mathfrak{C}(N)\leq1$.
\end{proof}

\begin{remark}
In fact, using the notation in the above proof, we also obtain:
\[
\lambda_N(A)
\begin{cases}
\leq 1/A\cdot\Gamma_1,&\quad \text{for}\ t=0;\\
\leq \gamma'_t/A\cdot \Gamma_{t+1},&\quad \text{for}\ 1\leq t\leq r-1;\\
=0,&\quad \text{for}\ t=r.
\end{cases}
\]
Here $\gamma'_t=1/\det(-\Gamma_i\cdot\Gamma_j)_{1\leq i,j\leq t}\, (\leq1/2)$.
\end{remark}

\begin{remark}\label{remark:K_F-N}
A classical example of the above situation is given by foliations.

Let $D=K_{\cF}=P+N$, where $\cF$ is a relatively minimal foliation of general type on a smooth surface $X$.
By \cite[Theorem~8.1]{Bru15} and \cite[Proposition~III.2.1]{McQ08}, the negative part decomposes as 
$N=\sum_{i=1}^s N_i$, where the $N_i$ are pairwise disjoint and each $N_i$
is a Hirzebruch--Jung chain as above.
In particular, $\mathfrak{C}(N)\le 1$.

Moreover, if $D=mK_{\cF}=P_m+N_m$, then $\mathfrak{C}(N_m)=m\,\mathfrak{C}(N)\le m$ by Lemma~\ref{lem:lambda-scaling}.
\end{remark}

\section{Comparison between Zariski and linear system decompositions}\label{sec:P+NvsM+Z}

In this section, we analyze the discrepancy between the Zariski decomposition
and the linear system decomposition of a big divisor.
\medskip

Let $D$ be a big divisor on a smooth projective surface $X$ with $h^0(D)\ge 2$.
Write the Zariski decomposition
\[
D = P + N,
\]
and the linear system decomposition
\[
|D| = |M| + Z.
\]

The main difficulty in comparing the Zariski decomposition
\(
D=P+N
\)
with the linear system decomposition
\(
|D|=|M|+Z
\)
is that the movable part \(M\) is generally not orthogonal to the negative part \(N\),
whereas the positive part \(P\) satisfies
\(
P\cdot \Gamma=0
\)
for every component \(\Gamma\subset \Supp(N)\).

The idea of this section is to modify \(M\) and \(Z\)
by adding suitable divisors supported on \(N\),
so that the corrected divisors become orthogonal to \(N\).
This leads to a precise comparison between \(P\) and \(M\),
measured by the invariant \(\lambda_N(M)\) which has a uniform control via $\mathfrak{C}(N)$ by Section~\ref{sec:N-nef}.
\medskip

We decompose $Z = Z_1 + Z_2$, where $Z_2$ is supported on $N$.

\subsection{Correction divisors and orthogonality}

Recall that for any $N$-nef divisor $A$, we associate a unique effective $\bQ$-divisor $E_N(A)$ supported on $N=\sum_i\gamma_i\Gamma_i$
such that
\[
E_N(A)\cdot \Gamma_i = -A\cdot \Gamma_i \quad \text{for all } i.
\]

This construction allows us to modify divisors by adding components supported on $N$
so that they become numerically orthogonal to $N$.
\medskip

We define the corrected divisors
\[
M^* := M + E_N(M), \qquad Z^* := Z_1 + E_N(Z_1).
\]
By construction, both $M^*$ and $Z^*$ satisfy
\[
M^*\cdot \Gamma_i = Z^*\cdot \Gamma_i = 0 \quad \text{for all } i.
\]

\begin{lemma}[Correction structure]\label{lem:Zariski correction structure}
We have the numerical equivalence
\[
P \equiv M^* + Z^*.
\]

Moreover:
\begin{enumerate}
\item $M^* \ge M \ge 0$, and $M=M^*$ if and only if $M\cdot N=0$.

\item $Z^* \ge 0$, and $Z^*=0$ if and only if ${\rm Supp}(Z)={\rm Supp}(N)$.

\item We have the decomposition
\[
Z = Z^* + N + E_N(M).
\]
In particular,
\[
Z=N \;\Longleftrightarrow\; \big\{\,Z^*=0 \ \text{and}\ M^*=M\,\big\}\;\Longleftrightarrow\; \big\{\,Z^*=0 \ \text{and}\ M\cdot N=0\,\big\}.
\]
\end{enumerate}
\end{lemma}

\begin{proof}
Note that
\[
P-M^*-Z^*\equiv Z_2-N-E_N(M)-E_N(Z_1)=:G.
\]
Here $G$ is a $\mathbb{Q}$-divisor supported on $\mathrm{Supp}(N)$, and
\[
G \cdot \Gamma_i = (P-M^*-Z^*)\cdot \Gamma_i = 0 \quad \text{for all } i.
\]
Since the intersection matrix on $\mathrm{Supp}(N)$ is negative definite, it follows that $G=0$.

This implies $P \equiv M^* + Z^*$ and
\[
Z=Z_1+Z_2=Z^*+(Z_2-E_N(Z_1))=Z^*+N+E_N(M).
\]
So (3) holds. Finally, (1) and (2) follow directly from the definitions of $M^*$ and $Z^*$.
\end{proof}

\subsection{Quantifying the discrepancy via $\lambda_N(M)$}
We now show that the discrepancy between \(P\) and \(M\)
is controlled by the quantity \(\lambda_N(M)\) introduced in Definition~\ref{def:lambda_N(A)}.
More precisely, we define
\[
\lambda_N(M):=\min \left\{ x \ge 0 \,\middle|\, (xE_N(M)-N)\cdot M=0 \right\}.
\]

The quantity \(\lambda_N(M)\) measures the interaction between the movable part \(M\)
and the negative part \(N\).
In particular, $\lambda_N(M)=0$ if and only if $M\cdot N=0$.

\begin{proposition}[Main identity]\label{prop:Mainidentity}
\begin{equation}\label{eq:main-identity}
P^2=M^2+\frac{1}{1+\lambda_N(M)}M\cdot Z+\frac{\lambda_N(M)}{1+\lambda_N(M)}M\cdot Z^*+P\cdot Z.
\end{equation}
\end{proposition}

\begin{proof}
Since
\[
 P\cdot M = D\cdot M - N\cdot M,
\qquad 
(M^*)^2 = M\cdot M^* = M^2 + E_N(M)\cdot M,
\]
we obtain
\begin{align*}
P\cdot M + \lambda_N(M)(M^*)^2
&= D\cdot M + \lambda_N(M)M^2 
+ \bigl(\lambda_N(M)E_N(M) - N\bigr)\cdot M\\
&= (1+\lambda_N(M))M^2+M\cdot Z,
\end{align*}
where we used the definition of $\lambda_N(M)$.

On the other hand,
\[
P^2 = P\cdot M + P\cdot Z,
\qquad
P^2 = (M^* + Z^*)^2 = (M^*)^2 + P\cdot Z + M\cdot Z^*.
\]
Combining these identities gives the result.
\end{proof}

The next proposition analyzes the vanishing of the correction terms
appearing in \eqref{eq:main-identity} and characterizes the corresponding rigidity cases.

\begin{proposition}\label{prop:Vanishing-and-rigidity}
Under the above notation, we have:
\begin{enumerate}
\item The following are equivalent:
\[
{\rm Supp}(Z)={\rm Supp}(N)
\;\Longleftrightarrow\;
Z^*=0
\;\Longleftrightarrow\;
M\cdot Z^*=0.
\]
In this case,
\[
P^2= M^2+\frac{1}{1+\lambda_N(M)}M\cdot Z.
\]

\item The following are equivalent:
\[
Z=N \;\Longleftrightarrow\; P^2 = M^2 \;\Longleftrightarrow\; M\cdot Z=0.
\]
\end{enumerate}
\end{proposition}

\begin{proof}
(1) The equivalence $Z^* = 0 \Longleftrightarrow {\rm Supp}(Z)={\rm Supp}(N)$ follows from Lemma~\ref{lem:Zariski correction structure}. We only prove $M\cdot Z^* = 0 \Longleftrightarrow Z^*=0$.
The implication $Z^*=0 \Rightarrow M\cdot Z^*=0$ is obvious.

Conversely, assume $M\cdot Z^*=0$. 
By the Hodge index theorem, $M^2\geq0$ implies $(Z^*)^2\leq 0$.
On the other hand, 
\[
(Z^*)^2=(P-M^*)\cdot Z^*=P\cdot Z^*-M\cdot Z^*=P\cdot Z^*\geq0.
\]
Hence $(Z^*)^2=P\cdot Z^*=0$. Since $P^2>0$, this implies $Z^*\equiv0$ by the Hodge index theorem again.
This implies $Z^*=0$ since $Z^*$ is effective.

\medskip

(2) If $Z=N$, then $P=M$ and $P^2=M^2$.
If $P^2=M^2$, Proposition~\ref{prop:Mainidentity} implies $M\cdot Z=0$.
Next it suffices to prove the implication $M\cdot Z=0 \Rightarrow Z=N$.

Assume $M\cdot Z=0$. Since $M$ is nef and 
\[
Z = N + Z^* + E_N(M)
\]
(cf.~Lemma~\ref{lem:Zariski correction structure}(3)),
$M\cdot Z=0$ implies
\[
M\cdot N = M\cdot Z^* = M\cdot E_N(M) = 0.
\]
In particular, $M\cdot N=0$ implies $E_N(M)=0$ and $M\cdot Z^* = 0$ implies $Z^*=0$ by (1).
Therefore, $Z=N$.
\end{proof}

\subsection{Uniform bounds via $\mathfrak{C}(N)$}

Recall that $\mathfrak{C}(N)$ is a numerical invariant depending only on $N$ (cf.~Definition~\ref{def:c(N)})
such that
\[
\lambda_N(M)\le \mathfrak{C}(N).
\]

Moreover, if $M \equiv nF$ for some divisor $F$ and $n\in\mathbb{Z}_{>0}$, then by Lemma~\ref{lem:lambda-scaling} we have
\[
\lambda_N(M)\le \frac{\mathfrak{C}(N)}{n}.
\]

In particular, the correction term $\lambda_N(M)$ is uniformly controlled by $N$
and decays linearly with respect to $n$.

\begin{remark}
When $n$ is large, the terms involving $\lambda_N(M)$ become negligible.
\end{remark}

As immediate consequences of Proposition~\ref{prop:Mainidentity} and \ref{prop:Vanishing-and-rigidity}, we have the following inequalities.
\begin{proposition}[Basic inequality]\label{prop:basic-ineq}
\begin{equation}\label{ineq:1st-ineq-of-P^2}
P^2 \ge M^2,
\end{equation}
with equality if and only if   $Z=N$.
\end{proposition}
\begin{proof}
It follows from Proposition~\ref{prop:Mainidentity} and Proposition~\ref{prop:Vanishing-and-rigidity}(2).
\end{proof}

\begin{proposition}[Refined inequality]\label{prop:refined-ineq}
Assume $M \equiv nF$ for some divisor $F$ and $n\in\bZ_{>0}$.
If $P^2> M^2$, then $F\cdot Z >0$ and
\begin{equation}\label{ineq:2nd-ineq-of-P^2}
P^2 \ge 
M^2 + \frac{n^2}{n+\mathfrak{C}(N)}\, F\cdot Z\geq M^2 + \frac{n^2}{n+\mathfrak{C}(N)},
\end{equation}
where the first inequality becomes an equality only if
${\rm Supp}(Z)={\rm Supp}(N)$.
\end{proposition}

\begin{proof}
Since $M \equiv nF$, we have
\[
M\cdot Z = n\, F\cdot Z.
\]
Moreover, by Lemma~\ref{lem:lambda-scaling}, we obtain
\[
\lambda_N(M)\le \frac{\mathfrak{C}(N)}{n}.
\]

Applying Proposition~\ref{prop:Mainidentity}, we deduce
\[
P^2 \ge 
M^2 + \frac{n}{n+\mathfrak{C}(N)}\, M\cdot Z
= M^2 + \frac{n^2}{n+\mathfrak{C}(N)}\, F\cdot Z.
\]
The first equality holds only if
${\rm Supp}(Z)={\rm Supp}(N)$, by Proposition~\ref{prop:Vanishing-and-rigidity}(1).
This yields the first inequality in \eqref{ineq:2nd-ineq-of-P^2}.

Next, assume $P^2 > M^2$. Then Proposition~\ref{prop:Vanishing-and-rigidity}(2)
implies that $M\cdot Z>0$, hence $F\cdot Z\ge 1$.
Substituting this into the previous inequality gives the second inequality in \eqref{ineq:2nd-ineq-of-P^2}.
\end{proof}

\section{Proof of the main Noether-type inequalities}\label{sec:proof}

In this section, we combine the inequalities established in
Section~\ref{sec:P+NvsM+Z} with the geometry of the movable linear system $|M|$
to establish the main Noether-type inequalities.

\medskip

Let $D$ be a big divisor on a smooth projective surface $X$ with $h^0(D)\geq2$.
Let $D=P+N$ be the Zariski decomposition of $D$, where $N$ denotes the negative part.
Let 
\[
|D| = |M| + Z
\]
be the linear system decomposition, where $M$ denotes the movable part.
The movable linear system $|M|$ defines a rational map
\[
\phi=\phi_{|D|}=\phi_{|M|}:X\dashrightarrow \bP^d,\qquad d=h^0(M)-1,
\]
where $h^0(M)=h^0(D)\geq2$.

To study the geometry of $|M|$, we resolve the indeterminacy of 
the rational map $\phi$ and consider the following diagram:
\begin{equation}\label{diag:main}
\xymatrix{
&&\widetilde{X}\ar_-{f}[lld]\ar^-{\psi}[d]\ar^-{\sigma}[rrr]&&&X\ar@{-->}^{\phi}[d]&\\
Y\ar^-{\pi}[rr]&&W\ar^-{\rho}_-{\rm desingularization}[rrr]&&&\Sigma\ar@{^{(}->}[r]&\bP^d
}
\end{equation}
where $\Sigma=\overline{{\rm Im}(\phi)}$.

Here $\sigma:\widetilde{X}\to X$ resolves the base points of $|M|$, and
$\rho:W\to\Sigma$ is the minimal desingularization. The morphism
$\rho\circ\psi:\widetilde{X}\to\Sigma\subset\bP^d$ is induced by the complete linear system $|\widetilde{M}|$, where
\[
\widetilde{M}=\sigma^*M-\sum_j a_j\mathcal{E}_j,\qquad a_j\ge1,
\]
and $\widetilde{M}\sim(\rho\circ\psi)^*H$ for any hyperplane section $H$ of $\Sigma$.
Moreover, $f:\widetilde{X}\to Y$  is obtained by the Stein factorization of $\psi$, satisfying $f_*\cO_{\widetilde{X}}=\cO_Y$.

\medskip

We distinguish two cases according to $\dim\Sigma$.

\subsection{The case $\dim\Sigma=1$}
We first consider the case $\dim\Sigma=1$.
The argument splits into two cases according to whether $M^2>0$ or $M^2=0$.

\begin{proposition}\label{prop:dim1}
Assume that $\dim\Sigma=1$ and $M^2>0$. Then we have
\begin{equation}\label{ineq:dim1-main}
{\rm vol}(D)\ge (h^0(D)-1)^2,
\end{equation}
with equality only if $Z=N$ and $Y\cong\bP^1$.

Moreover, if equality does not hold, then
\begin{equation}\label{ineq:dim1-refined}
{\rm vol}(D)\ge (h^0(D)-1)^2+\frac{(h^0(D)-1)^2}{h^0(D)-1+\mathfrak{C}(N)},
\end{equation}
with equality only if ${\rm Supp}(Z)={\rm Supp}(N)$ and $Y\cong\bP^1$.
\end{proposition}

\begin{proof}
In this case, $f:\widetilde{X}\to Y$ is a fibration onto a smooth curve.
Let $H$ denote a hyperplane section of $\Sigma$ and let $L:=(\rho\circ \pi)^*(H)$. Then
\[
\widetilde{M}\sim(\rho\circ\psi)^*(H)=(\rho\circ\pi\circ f)^*(H)=f^*(L).
\]
Since $f_*\cO_{\widetilde{X}}=\cO_Y$, we have $f_*\widetilde{M}=L$ and hence
\[
h^0(Y,L)=h^0(\widetilde{X},\widetilde{M})=h^0(X,M).
\]
By Lemma~\ref{lem:degD>=h0(D)-1-curve}, we have 
$n:=\deg L\geq h^0(M)-1$, with equality if and only if $Y\cong\bP^1$.

Let $F$ be a general fibre of $f$. Then
$\widetilde{M}\equiv nF$.
Let $A=\sigma_*F$. Then 
\[
M\equiv nA.
\]
Since $M^2>0$, we have $A^2>0$. Hence
\begin{equation}\label{ineq:M^2>=}
M^2=n^2A^2\geq n^2\ge (h^0(M)-1)^2.
\end{equation}

Combining this with Proposition~\ref{prop:basic-ineq}, we obtain
\[
P^2\ge M^2\ge (h^0(M)-1)^2,
\]
which proves \eqref{ineq:dim1-main}.
Moreover, equality $P^2=(h^0(M)-1)^2$ holds if and only if
\[
P^2=M^2
\qquad\text{and}\qquad
M^2=(h^0(M)-1)^2.
\]
By Proposition~\ref{prop:Vanishing-and-rigidity}, the first condition is equivalent to $Z=N$.
The second condition forces
$n=h^0(M)-1$,
and hence $Y\cong\bP^1$.
This proves the equality characterization.

\medskip

Now assume that equality in \eqref{ineq:dim1-main} does not hold. Then 
\[
P^2>M^2\quad \text{or}\quad M^2>(h^0(M)-1)^2.
\]
If $P^2>M^2$, then Proposition~\ref{prop:refined-ineq} gives
\[
P^2
\ge
M^2+\frac{n^2}{n+\mathfrak{C}(N)}\ge
(h^0(M)-1)^2
+
\frac{(h^0(M)-1)^2}{h^0(M)-1+\mathfrak{C}(N)}.
\]
If $M^2>(h^0(M)-1)^2$, then 
\[
P^2\geq M^2\geq (h^0(M)-1)^2+2(h^0(M)-1)+1,
\]
where the last inequality follows from \eqref{ineq:M^2>=}.

As 
\[
2(h^0(M)-1)+1-\frac{(h^0(M)-1)^2}{h^0(M)-1+\mathfrak{C}(N)}\geq h^0(M)>0,
\]
we obtain
\[
P^2\ge  (h^0(M)-1)^2+\frac{(h^0(M)-1)^2}{h^0(M)-1+\mathfrak{C}(N)}.
\]
This proves \eqref{ineq:dim1-refined}.
Moreover, if equality holds, then ${\rm Supp}(Z)={\rm Supp}(N)$ and $Y\cong\bP^1$,
by Proposition~\ref{prop:refined-ineq} and the above discussion.
\end{proof}

\medskip

We now consider the case where $M^2=0$.

\begin{proposition}\label{prop:dim1-bpf}
Assume that $\dim\Sigma=1$ and $M^2=0$. Then $X=\widetilde{X}$ and
\begin{equation}\label{ineq:dim1-bpf}
{\rm vol}(D)\ge 
\frac{(h^0(D)-1)^2}{h^0(D)-1+\mathfrak{C}(N)}\, D\cdot F
\ge 
\frac{(h^0(D)-1)^2}{h^0(D)-1+\mathfrak{C}(N)}.
\end{equation}
Moreover, the first equality in \eqref{ineq:dim1-bpf} holds  only if 
${\rm Supp}(Z)={\rm Supp}(N)$ and $Y\cong\bP^1$.
\end{proposition}

\begin{proof}
In this case, $|M|$ is base-point-free. 
Then $X=\widetilde{X}$ and $M=\widetilde{M}\equiv nF$, where $n$ is defined in the proof of Proposition~\ref{prop:dim1}. 
Moreover, $n\ge h^0(M)-1$ with equality iff $Y\cong\bP^1$, by Lemma~\ref{lem:degD>=h0(D)-1-curve}.

Since $M^2=0$ and $P^2>0$, by Proposition~\ref{prop:refined-ineq}, we obtain $D\cdot F=F\cdot Z\geq1$ and
\[
P^2\geq \frac{n^2}{n+\mathfrak{C}(N)}F\cdot Z\geq \frac{(h^0(M)-1)^2}{h^0(M)-1+\mathfrak{C}(N)}D\cdot F,
\]
This gives the first inequality in \eqref{ineq:dim1-bpf}, and equality holds only if 
${\rm Supp}(Z)={\rm Supp}(N)$ and $Y\cong\bP^1$.
The second inequality in \eqref{ineq:dim1-bpf} 
follows from the fact that $D\cdot F\geq1$.
\end{proof}

\begin{proof}[Proof of Theorem~\ref{thm:main}(1)]
It follows from Proposition~\ref{prop:dim1} and Proposition~\ref{prop:dim1-bpf}.
\end{proof}
%-----------------------------------
\subsection{The case $\dim\Sigma=2$}
We derive lower bounds for $M^2$ from the geometry of the image surface $\Sigma$.
These will then be combined with the comparison results in Section~\ref{sec:P+NvsM+Z}.

\begin{lemma}\label{lem:M^2-dimS=2}
Assume that $\dim\Sigma=2$.
\begin{enumerate}
\item We have
\[
M^2\ge h^0(M)-2,
\]
with equality if and only if $\phi$ is a birational morphism onto its image and $\deg\Sigma=d-1$.

\item If moreover $\kappa(X)\ge0$, then
\[
M^2\ge 2h^0(M)-4,
\]
with equality if and only if one of the following holds:
\begin{itemize}
    \item $\phi$ is a birational morphism onto its image and $\deg\Sigma=2d-2$;
    
    \item $\phi$ is a finite morphism of degree $2$ onto its image and $\deg\Sigma=d-1$.
\end{itemize}
\end{enumerate}
\end{lemma}

\begin{proof}
Since $\dim\Sigma=2$, we have
\[
M^2\ge \widetilde{M}^2
=
((\rho\circ\psi)^*H)^2
=
\deg(\rho\circ\psi)\cdot H^2
=
\deg\pi\cdot\deg\Sigma.
\]

By Proposition~\ref{Prop:Bea79Lem1.4}(1),
\[
\deg\Sigma\ge d-1=h^0(M)-2.
\]
Hence
\[
M^2\ge h^0(M)-2,
\]
with equality if and only if
\[
M^2=\widetilde{M}^2,
\qquad
\deg\pi=1,
\qquad
\deg\Sigma=d-1.
\]
The equality $M^2=\widetilde{M}^2$ holds if and only if $|M|$ is base point free, 
while $\deg\pi=1$ is equivalent to $\phi$ being birational onto its image.
Thus equality holds if and only if $\phi$ is a birational morphism onto its image and $\deg\Sigma=d-1$.
This proves~(1).

\medskip

Now assume moreover that $\kappa(X)\ge0$.

If $\deg\pi=1$, then $\phi$ is birational onto its image, and hence
$\kappa(\Sigma)=\kappa(X)\ge0$.
By Proposition~\ref{Prop:Bea79Lem1.4}(2), we have
\[
\deg\Sigma\ge 2d-2=2h^0(M)-4.
\]
Therefore
\[
M^2\ge 2h^0(M)-4,
\]
with equality if and only if $\phi$ is a birational morphism onto its image and $\deg\Sigma=2d-2$.

We now consider the case $\deg\pi\ge2$.
 Then
\[
M^2\ge\widetilde{M}^2=\deg\pi\cdot\deg\Sigma
\ge 2\deg\Sigma
\ge 2h^0(M)-4.
\]
Equality $M^2=2h^0(M)-4$ holds if and only if
\[
M^2=\widetilde{M}^2,
\qquad\deg\pi=2,
\qquad
\deg\Sigma=d-1.
\]
Equivalently, $\phi$ is a finite morphism of degree $2$ onto its image and $\deg\Sigma=d-1$.
This proves~(2).
\end{proof}

\begin{proof}[Proof of Theorem~\ref{thm:main}(2)(3)]
We first prove Theorem~\ref{thm:main}(2).

Combining Lemma~\ref{lem:M^2-dimS=2}(1) with Proposition~\ref{prop:basic-ineq}, we obtain
\[
P^2\ge M^2\ge h^0(M)-2,
\]
which yields \eqref{ineq:P^2-intro-dim2-general-1st}.
Moreover, equality
$P^2=h^0(M)-2$
holds if and only if
\[
P^2=M^2
\qquad\text{and}\qquad
M^2=h^0(M)-2.
\]
By Proposition~\ref{prop:Vanishing-and-rigidity}(2), the first condition is equivalent to
$Z=N$.
By Lemma~\ref{lem:M^2-dimS=2}(1), the second condition is equivalent to
$\phi$ being a birational morphism onto its image and
$\deg\Sigma=d-1$.
Finally, Proposition~\ref{prop:degS=n-1} shows that $\Sigma$ is a normal rational surface.
This proves the equality characterization in Theorem~\ref{thm:main}(2).

\medskip

Now assume that $P^2>h^0(M)-2$. 
Then 
\[
P^2>M^2\qquad \text{or}\qquad M^2\geq h^0(M)-1.
\]
If $M^2\geq h^0(M)-1$, then 
\[
P^2\geq M^2\geq h^0(M)-2+1.
\]
If $P^2>M^2$, then Proposition~\ref{prop:refined-ineq} yields
\[
P^2
\ge
M^2+\frac{1}{1+\mathfrak{C}(N)}
\ge
h^0(M)-2+\frac{1}{1+\mathfrak{C}(N)}.
\]
Thus,
\[
P^2\geq h^0(M)-2+\min\left\{\,1,\,\frac{1}{1+\mathfrak{C}(N)}\,\right\}=h^0(M)-2+\frac{1}{1+\mathfrak{C}(N)},
\]
which proves \eqref{ineq:P^2-intro-dim2-general-2nd}.
\bigskip

We now turn to the proof of Theorem~\ref{thm:main}(3).

Assume moreover that $\kappa(X)\ge0$.
Combining Lemma~\ref{lem:M^2-dimS=2}(2) and Proposition~\ref{prop:basic-ineq}, we obtain
\[
P^2\ge M^2\ge 2h^0(M)-4,
\]
which proves \eqref{ineq:P^2-intro-dim2-not-rule-1st}.
Moreover, equality $P^2=2h^0(M)-4$
holds if and only if
\[
P^2=M^2
\qquad\text{and}\qquad
M^2=2h^0(M)-4.
\]
By Proposition~\ref{prop:Vanishing-and-rigidity}(2), the first condition is equivalent to
$Z=N$.
By Lemma~\ref{lem:M^2-dimS=2}(2), the second condition is equivalent to one of the following cases:
\begin{itemize}
\item $\phi$ is a birational morphism onto its image and $\deg\Sigma=2d-2$;

\item $\phi$ is a finite morphism of degree $2$ onto its image and $\deg\Sigma=d-1$.
\end{itemize}
In the first case, Proposition~\ref{prop:degS=2d-2} implies that
$\Sigma$ is birational to a K3 surface.
In the second case, Proposition~\ref{prop:degS=n-1} implies that
$\Sigma$ is a normal rational surface.
This proves the equality characterization in
Theorem~\ref{thm:main}(3).

\medskip

Finally, if $P^2>2h^0(M)-4$, then
\[
P^2\geq 2h^0(M)-4+\min\left\{\,1,\,\frac{1}{1+\mathfrak{C}(N)}\,\right\}
=
2h^0(M)-4+\frac{1}{1+\mathfrak{C}(N)},
\]
where the argument is identical to that used in the proof of
\eqref{ineq:P^2-intro-dim2-general-2nd}.
This proves \eqref{ineq:P^2-intro-dim2-not-rule-2nd}.
\end{proof}

\subsection{Proof of Corollary~\ref{coro:vol-ps-index}}
\begin{proof}
Let $m=\iota(D)$, so that $h^0(mD)\ge 2$.
Write the Zariski decomposition
\[
mD = P' + N',
\]
so that $N' = mN$ and $\mathfrak{C}(N') = m\mathfrak{C}(N)$ by Lemma~\ref{lem:lambda-scaling}.

If $\dim {\rm Im}\phi_{|mD|} = 1$, then by Theorem~\ref{thm:main}(1) we have 
\[
\mathrm{vol}(mD)
\ge
\frac{(h^0(mD)-1)^2}{h^0(mD)-1+\mathfrak{C}(N')}
\ge
\frac{1}{1+m\mathfrak{C}(N)}.
\]

If $\dim {\rm Im}\phi_{|mD|} = 2$, then $h^0(mD)\ge 3$, and by Theorem~\ref{thm:main}(2) we obtain
\[
\mathrm{vol}(mD) \ge h^0(mD)-2 \ge 1.
\]

In both cases,
\[
\mathrm{vol}(mD) \ge \frac{1}{1+m\mathfrak{C}(N)}.
\]
Since $\mathrm{vol}(mD) = m^2 \mathrm{vol}(D)$, we conclude that
\[
\mathrm{vol}(D)
\ge
\frac{1}{m^2}\cdot \frac{1}{1+m\mathfrak{C}(N)},
\]
which proves the claim.
\end{proof}

\bigskip

\subsection*{Acknowledgements}
The author sincerely thanks Professors Shengli Tan, Jun Lu, Xiaohang Wu, and Xin L\"u for their helpful discussions and valuable suggestions.  
He also thanks the referees for carefully reading the manuscript, pointing out errors, and providing detailed comments that significantly improved the paper.

% ==================== Bibliography ====================

\end{document}